\documentclass[12pt]{amsart}
\usepackage{templatestyle}
\usepackage{templatecommand}
\usepackage{templatetensor}
\allowdisplaybreaks[3]
\title[Schwarz lemma for conical K\"ahler metrics]
{Schwarz lemma for conical K\"ahler metrics with general cone angles} 

\author
{Ryosuke Nomura}

\address{Graduate School of Mathematical Sciences, The University of Tokyo \endgraf
	3-8-1 Komaba, Meguro-ku, Tokyo, 153-8914, Japan.}

\email{nomu@ms.u-tokyo.ac.jp}

\thanks{Classification AMS 2010: 53C55, 
	32W20.
}


\begin{document}
	
\begin{abstract}
	The Schwarz--Pick lemma is a fundamental result in complex analysis. 
	It is well-known that Yau generalized it to the higher dimensional manifolds by applying his maximum principle for complete Riemannian manifolds. 
	Jeffres obtained Schwarz lemma for volume forms of conical K\"ahler metrics, based on a barrier function and the maximum principle argument. 
	In this note, we generalize Jeffres' result to general cone angles including the case when the pullback of the metric would blows up along the divisors.
\end{abstract}
\maketitle 
	
\section{Introduction}

The Schwarz--Pick lemma states that any holomorphic map between the unit disks in the complex plane decreases the Poincar\'e metrics. 
After that, Ahlfors \citep{Ahlfors39Schwarz} generalized it to a holomorphic map from the unit disk to a hyperbolic Riemann surface.
For higher dimensions, Yau \citep{Yau78Schwarz} showed that any holomorphic map from complete \kahler \ manifold whose Ricci curvature is bounded from below to a Hermitian manifold whose holomorphic bisectional curvature is bounded by a negative constant decreases the metric up to a multiplicative constant. Also, he showed that, under similar conditions on curvatures, any holomorphic map decreases the volume forms up to a multiplicative constant. 
Both results essentially based on his maximum principle for complete Riemannian manifolds. 
Later on, many generalizations obtained in various geometric settigs. 

In this note, we forcus on the conical K\"ahler metrics, for short, cone metrics. 
Let $X$ be a compact \kahler \ manifold of dimension $n$, $D$ be a smooth divisor on $X$, and $\beta $ be a real number satisfying $0< \beta<1$. The cone metric $\omega$ with cone angle $2\pi \beta $ along $D$ is a \kahler \ metric on $X\setminus D$ which is locally quasi-isometric to the standard cone metric
\begin{align*}
	\omega_\beta  \deq \dfrac{\beta ^2}{|z|^{2(1-\beta )}}\dfrac{\ii}{2} dz^1 \wedge d\zob  + \sum_{i=2}^n \dfrac{\ii}{2}  dz^i \wedge d\zib ,
\end{align*}
and satisfies some regularity conditions. (For a precise definition of the cone metric, see Definition \refs{ConeDef}.)
The notion of cone metrics plays an important role in recent advances in \kahler \ geometries, in particular \kahler -Einstein problems, for instance see [\citet{CDSI}, \citet{CDSII}, \citet{CDSIII}, \citet{Tian}].

To state the theorems, we use the following setups and notations. 
\begin{setups}
Let $X$ and $Y$ be compact \kahler \ manifolds, $D\subset X$, $E\subset Y$ be smooth divisors, and $f \colon X \rightarrow Y$ be a surjective holomorphic map satisfying $f^{\ast}(E) =k D$ with $k \in \zp$. 
Let $\omx$ (resp. $\omy$) be a cone metric with cone angle $2\pi \alpha $ (resp. $2\pi \beta$) along $D$ (resp. $E$) on $X$ (resp. $Y$). 
Let $s \in H^0(X, \mathcal{O}_X(D))$ be a holomorphic section of the line bundle $\mathcal{O}_X(D)$ whose zero divisor is $D$ and $h$ be a smooth Hermitian metric on it satisfying $|s|_h \le 1$. Let $C>0$ be  an upper bound for the Chern curvature of $h$ i.e.  $\ii R_h \le C \omx$. 
For a \kahler \ form $\omega$, we will denote by  
$\ric (\omega)$ the Ricci curvature of $\omega$, 
$R(\omega)$ the scalar curvature of $\omega$, 
and $\bisect(\omega)$ the bisectional curvature of $\omega$.
\end{setups}

Schwarz lemma for the cone metrics obtained by Jeffres \citep{Jeffres00} is states as follows. 
\begin{theorem}[{\citep[Theorem]{Jeffres00}}]\label{JeffresVolumeThm}
Assume that  $\dim X = \dim Y = n$, the cone angles satisfy $\alpha \le \beta $ and there exists non-negative constants $A, B\ge  0$ satisfying
\begin{align}\label{CurvVolineq}
	R(\omx) \ge -A, \ \ric(\omy) \le -B \omy<0 .
\end{align} 
Then, the volume forms satisfy
\begin{align*}
	\fomy^n \le \left( \dfrac{A}{B} \right)^n \omx^n \son X\setminus D.
\end{align*}
\end{theorem}
Since the cone metric is not complete on $X \setminus D$, we cannot apply the maximum principle argument directly. Jeffers overcame this difficulty by using a barrier function, called ``Jeffres' trick''.  
However, his original proof seems to need more assumptions on the regularity of the cone metrics along $D$ as in Definition \refs{Donaldson}  (see the proof of Proposition \ref{JeffresTrick}).

In this note, we will generalize this theorem to a general cone angle and prove a Schwarz lemma for cone metrics.

\begin{theorem}[Volume forms]\label{VolThm}
Assume that $\dim X = \dim Y =n$ and the curvature condition (\refs{CurvVolineq}) holds. 
\begin{enua}
	\item Suppose $\alpha \le k\beta $. Then we have 
	\begin{align*}
		\fomy^n \le \left( \dfrac{A}{nB} \right)^n \omx^n \son X\setminus D.
	\end{align*}
	\item Suppose $\alpha > k\beta $. Then we have 
			\begin{align*}
					\fomy^n \le \left( \dfrac{A+ (\alpha - k\beta ) C}{nB} \right)^n \dfrac{\omx^n}{|s|_h ^{2(\alpha -k\beta)}} \son X\setminus D.
			\end{align*}
\end{enua}
\end{theorem}

We remark that the condition $\alpha \le k\beta $ on cone angles in the statement (a) is weaker than assumptions in Theorem \refs{JeffresVolumeThm}.
\begin{theorem}[Metrics]\label{TraceThm}
Assume that there exists non-negative constants $A, B\ge  0$ such that the curvatures satisfy the following: 
\begin{align}\label{CurvTrineq}
	\ric(\omx) \ge -A \omx, \ \bisect (\omy) \le -B <0 .
\end{align}
\begin{enua}
	\item Suppose $\alpha \le k\beta $. Then we have 
	\begin{align*}
		\fomy \le \dfrac{A}{B}\omx \son X\setminus D.
	\end{align*}
	\item Suppose $\alpha > k\beta $. Then we have 
		\begin{align*}
				\fomy \le  \dfrac{A+ (\alpha - k\beta ) C}{B}  \dfrac{\omx}{|s|_h ^{2(\alpha -k\beta)}} \son X\setminus D.
		\end{align*}
\end{enua}
\end{theorem}

If the cone angle satisfies $\alpha > k \beta$, the pullback $\fomy$ has singularites along $D$. 
In fact, even in a one-dimensional case, the pullback of the standard cone metric $\omega_\beta = (\beta^2 /|w|^{2(1-\beta )})\ii dw\wedge d\overline{w}/2$ by $f: z\mapsto w=z^k$ is given by
\begin{align*}
	f^\ast \omega_\beta = \beta^2 k^2 |z|^{2(k\beta-1)}\dfrac{\ii }{2} dz \wedge d\overline{z},
\end{align*}
therefore we have
\begin{align*}
	\dfrac{f^\ast \omega_\beta}{\omega_\alpha }= \dfrac{\beta^2 k^2}{\alpha^2 }|z|^{2(k\beta-\alpha )},
\end{align*}
which is singular if $\alpha > k \beta$. 

\vspace{10pt}

\noindent 
\textbf{Acknowledgment. }The author would like to thank his supervisor Prof. Shigeharu Takayama for various comments.
This work is supported by the Program for Leading Graduate Schools, MEXT, Japan.

\section{Cone metrics}

In this section, we recall the definition of cone metrics following \citep[Section 4]{Donaldson12KametricConesing}. Let $X$ be a compact \kahler \ manifold of dimension $n$, $D$ be a smooth divisor on $X$, and $\beta $ be a real number satisfying $0< \beta<1$. 
We first remark that if we take a local holomorphic chart $(U, (z^1, \dots, z^n))$ satisfying $D\cap U =\{ z^1 =0 \}$, 
the standard cone metric $\omega_\beta $ induces a distance function $d_\beta $ on $U$ which is expressed as 
\begin{align*}
	d_\beta (z,w) = \left( 
	\left|  (z^1)^\beta - (w^1)^\beta \right|^2
	+ | z^2 - w^2 |^2 + \cdots + | z^n - w^n |^2 
 \right)^{1/2},
\end{align*}
where $z=(z^1,\dots, z^n), w=(w^1, \dots, w^n)$. Here, we take a suitable branch of $z^\beta$.

\begin{definition}[$C^{2, \alpha ,\beta}$-functions]\label{Donaldson}
Let $\alpha $ be a constant satiftying $0< \alpha < \min\{ 1/\beta -1, 1  \}$. We define the regularites of functions along $D$ as follows. 
\begin{enumerate}
	\item A function $f$ on $X$ is said to be {\it of class $C^{,\alpha ,\beta}$} if for any local holomorphic chart $(U, (z^1, \dots, z^n))$ satisfying $D\cap U =\{ z^1 =0 \}$, 
	$f$ is  an $\alpha$-\holder \ function on $U$ with respect to the distance function $d_\beta $.
	
	This definition is equivalent to the following statement which is the original definition in \citep{Donaldson12KametricConesing}. We set $\widetilde{f}$ by $\widetilde{f}(\xi , z^2, \dots, z^n) \deq f ( |\xi |^{1/\beta -1}\, \xi ,$ $z^2, \dots, z^n)$. Then $\widetilde{f}$ is an $\alpha$-\holder \ function with respect to $\xi , z^2, \dots, z^n$ with respect to the Euclidean distance.
	\item A $(1, 0)$-form $\tau $ is said to be {\it of class $C^{,\alpha ,\beta}$} if 
				\begin{align*}
					\left| z^{1} \right| ^{1-\beta} \tau \left(  \delzo  \right) & \in C^{,\alpha ,\beta}, \\
					\tau \left( \delzi   \right) &  \in C^{,\alpha ,\beta} \sfor i=2, \dots, n
				\end{align*}
	\item A $(1,1)$-form $\sigma $ is said to be {\it of class $C^{,\alpha ,\beta}$} if 
				\begin{align*}
					\left| z^{1} \right| ^{2(1-\beta)} \sigma \left(  \delzo,  \delzob   \right) &  \in C^{,\alpha ,\beta}, \\
					\left| z^{1} \right| ^{1-\beta} \sigma \left( \delzo,  \delzib  \right) &  \in C^{,\alpha ,\beta} \sfor i=2, \dots, n, \\
					\left| z^{1} \right| ^{1-\beta} \sigma \left( \delzi,  \delzob  \right) &  \in C^{,\alpha ,\beta} \sfor i=2, \dots, n, \\
					\sigma \left( \delzi,  \delzjb  \right) &  \in C^{,\alpha ,\beta} \sfor i,j=2, \dots, n.
				\end{align*}
	\item A function $f$ is said to be {\it of class $C^{2,\alpha ,\beta}$} if $f$, $\partial f$, $\dbar f$, $\ddb f$ are of class $C^{,\alpha ,\beta}$.
\end{enumerate}
\end{definition}
\vspace{10pt}

\begin{definition}[Cone metrics]\label{ConeDef}
A closed positive $(1,1)$-current $\omega$ on $X$ is called a {\it cone metric with cone angle $2\pi \beta $ along $D$} if it satisfies the following three conditions:
\begin{enui}
	\item $\omega$ is a \kahler \ metric on $X\setminus D$
	\item For each point $x \in D$, there exists a local holomorphic chart $(U, (z^1, \dots, z^n))$ satisfying $D\cap U =\{ z^1 =0 \}$ such that $\omega $ is quasi-isometric to $\omega_\beta$ on $U\setminus D$, that is, there exists a constant $C= C_U >0$ such that
		\begin{align*}
			\dfrac{1}{C} \omega_\beta \le \omega \le C \omega_\beta \son U \setminus D.
		\end{align*}
	Here, $\omega_\beta $ is the standard cone metric defined by 
					\begin{align*}
						\omega_\beta  \deq \dfrac{\beta ^2}{|z|^{2(1-\beta )}}\dfrac{\ii}{2} dz^1 \wedge d\zob  + \sum_{i=2}^n \dfrac{\ii}{2}  dz^i \wedge d\zib.			
					\end{align*}
	\item There exists a smooth \kahler \ form $\omega_0$ on $X$, and a $ C^{2,\alpha ,\beta}$-function $ \varphi $ such that 
			\begin{align*}
					\omega = \omega_0 + \ddb \varphi .
				\end{align*}	
\end{enui}
\end{definition}
In \citep{Jeffres00}, the regularity condition (iii) does not assumed. However, we assume here.

A typical example of the cone metric is $\omega \deq \omega _0 + \delta \ddb |s|_h^{\beta}$, where $\omega_0$ is a smooth \kahler \ metric on $X$, $\delta $ is a sufficiently small constant, $s \in H^0(X, \mathcal{O}_X(D))$ is a holomorphic section of the line bundle $\mathcal{O}_X(D)$ whose zero divisor is $D$, and $h$ is a smooth Hermitian metric. 

\section{Proof of the theorems}

To prove the theorem, we need the following Laplacian estimates which are obtained by [\citet{Chern68}, \citet{Lu68}]. For the readers convenience, we prove here.

\begin{proposition}\label{LapProp}
Let $X,Y$ be (not necessarily compact) \kahler \ manifolds, and $f \colon X \rightarrow Y$ be a holomorphic map. Let $\omx$ (resp. $\omy$) be a smooth \kahler \ metric on $X$ (resp. $Y$). 
We set $v \deq \fomy^n/\omx^n$, and $u \deq \tromx{\fomy}$.
\begin{enua}
	\item Suppose that there exists non-negative constants $A, B\ge  0$ satisfying
$
		R(\omx) \ge -A, \ \ric(\omy) \le -B \omy, 
$
	and $\dim X = \dim Y = n$. Then we have
		\begin{align*}
			\lapomx \log v 
				&\ge nB  v^{1/n} - A, \\ 
			\lapomx v 
				&\ge v ( nB  v^{1/n} - A  ).
		\end{align*}
	\item Suppose that there exists non-negative constants $A, B\ge  0$ satisfying $\ric (\omx) \ge -A\omx$, $\bisect ( \omy ) \le -B \omy$.
		Then we have
		\begin{align*}
			\lapomx \log  u 
				&\ge Bu-A, \\
			\lapomx u 
				&\ge u (Bu-A ).
		\end{align*}		
\end{enua}
\end{proposition}
\vspace{10pt}
\begin{proof}
Let $(z^1, \dots, z^n)$ and $(w^1, \dots, w^n)$ be normal coodinates on $X$ and $Y$ respectively. We set
\begin{align*}
	\omx = \ii \gijb \dzidzjb , \ \omy = \ii \habb \dwadwbb.
\end{align*}
\noindent
(a) 
$v$ is locally denoted as 
\begin{align}\label{vollocal}
	v = \dfrac{\fomy^n}{\omx^n} = \dfrac{\det(\habb \circ f )\,  |\det J(f)|^2}{\det (\gijb )} 
\end{align}
where $J(f)$ is the  Jacobian of $f$. Therefore, on $\Omega\deq \{ x\in X \mid \det J(f)(x)\neq 0  \}$, we obtain
\begin{align*}
	\ddb \log v
		&=  f^\ast \ddb \log \det(\habb) + \ddb \log \det (\gijb ) - \ddb \log |\det J(f)|^2 \\[5pt]
		&= f^\ast (- \ric (\omy ) ) + \ric (\omx ) .
\end{align*}
By the assumption on curvatures and the inequality of arithmetic and geometric means, we have the following estimates on $\Omega$: 
\begin{align*}
	\lapomx \log v 
		&= \tromx{ \ddb \log v  } =\tromx{ f^\ast (- \ric (\omy ) )} + R (\omx ) \\
		&\ge B \tromx{f^\ast \omy } - A \\[5pt]
		&\ge nB  v^{1/n} - A, \\[5pt]
	\lapomx v
		&= \lapomx e^{\log v} = e^{\log v} \left( |\nabla \log v|_{\omx}^2 + \lapomx \log v \right)\\[5pt]
		&\ge v \lapomx \log v \\[5pt]
		&\ge v (nB v^{1/n} - A ) .
\end{align*}
 By continuity, the last inequality holds on the whole $X$.  \\

\noindent 
(b) We set
\begin{align*}
	f^\ast \omy = \ii \hijb \dzidzjb \deq  \ii (\habb\circ f )  (\deli f^\alpha )  \overline{(\delj f^\beta)}\dzidzjb,
\end{align*}
and denote $?R_i \jbar k \lbar ?$ and $?S_\alpha \bbar \gamma \dbars ? $ by the curvature tensor of $\omx$ and $\omy$ respectively.
Then we have the following inequalities, which are our assertion (b).
\begin{align*}
	&\lapomx \tromx{\fomy }
		= \guklb \delk \dellb \left( \guijb \hijb \right)
		=\guklb \left( \delk \dell \guijb  \right)\hijb
				+ \guklb \guijb \left(   \delk \dell \hijb  \right)\\[2pt]
		&=\guklb ?R_k \lbar ^i\jbar ? \hijb 
				+\Bigl(  
						\guklb \guijb (\deli \delk f^\alpha ) \overline{(\delj \dell f^\beta)}
		- \guklb \guijb (\deli f^\alpha )  \overline{(\delj f^\beta)} (\delk f^\gamma )  \overline{(\dell f^\delta )} ? S_ \alpha \bbar \gamma \dbars? 
				 \Bigr)\\[2pt]
		&\ge \angomx{\ric(\omega), \fomy}-  (\deli f^\alpha )  \overline{(\deli f^\beta)} (\delk f^\gamma )  \overline{(\delk f^\delta )} ?S_ \alpha \bbar \gamma \dbars? \\[2pt]
		&\ge -Au + B u^2 = u (Bu -A ) .\\[2pt]
	&\lapomx \log \tromx{\fomy}
		= \dfrac{\lapomx \tromx{ \fomy }}{ \tromx{\fomy} } -\dfrac{ \absom{\nabla \tromx{\fomy}}^2 }{(\tromx{\fomy})^2}\\[2pt]
		&=\dfrac{1}{ \tromx{\fomy} }
				\Bigl( 
					\angomx{\ric(\omx), \fomy}
					- \guklb \guijb (\deli f^\alpha )  \overline{(\delj f^\beta)} (\delk f^\gamma )  \overline{(\dell f^\delta )} ? S_ \alpha \bbar \gamma \dbars? 
					\\[2pt]
			&\quad+ \guklb \guijb (\deli \delk f^\alpha ) \overline{(\delj \dell f^\beta)}
							\Bigr) 	
						 -\dfrac{ \absomx{\nabla \tromx{\fomy}}^2 }{(\tromx{\fomy})^2}	\\[2pt]
		&=\dfrac{1}{ \tromx{\fomy} }
				\left( 
					\angomx{\ric(\omx), \fomy}
					- \guklb \guijb (\deli f^\alpha )  \overline{(\delj f^\beta)} (\delk f^\gamma )  \overline{(\dell f^\delta )} ? S_ \alpha \bbar \gamma \dbars? 
			\right) \\[2pt]
		&\quad +\dfrac{1}{(\tromx{\fomy})^2}	
				\left(
						\tromx{\fomy} \guklb \guijb (\deli \delk f^\alpha ) \overline{(\delj \dell f^\beta)}
						-\absomx{\nabla \tromx{\fomy}}^2
				\right)\\[2pt]
		&\ge\dfrac{1}{ \tromx{\fomy} }
				\left( 
					\angomx{\ric(\omx), \fomy}
					- \guklb \guijb (\deli f^\alpha )  \overline{(\delj f^\beta)} (\delk f^\gamma )  \overline{(\dell f^\delta )} ? S_ \alpha \bbar \gamma \dbars? 
				\right) \\[2pt]
		&\ge Bu -A.
\end{align*}
In the second line from the bottom, we used the following inequality:
\begin{align*}
\absomx{\nabla \tromx{\fomy}}^2
	&= \guijb ( \deli \guklb \hklb )(\deljb \gupqb \hpqb ) = \guijb\guklb  \gupqb ( \deli  \hklb   )(\deljb  \hpqb )\\[5pt]
	&=\sum_{i,k,p, \alpha ,\beta }  (\deli \delk f^\alpha )\overline{(\delk f^\alpha )}  \overline{(\deli \delp f^\beta ) }  (\delp f^\beta )\\
	&\le \sum_{k,p, \alpha ,\beta } 
		\left(
			| \delp f^\beta | | \delk f^\alpha | 
			 \left( \sum_i |\deli \delk f^\alpha |^2  \right)^{1/2} 
			 \left( \sum_j |\delj \delp f^\beta |^2 \right)^{1/2} 
			 	\right) \\
	&=\left( \sum_{k, \alpha }  | \delk f^\beta | \left( \sum_i |(\deli \delk f^\alpha |^2  \right)^{1/2}  \right)^2\\
	&\le \left( \sum_{l, \beta } | \dell f^\beta | ^2 \right)
			\left(  \sum_{i,k ,\alpha }|\deli \delk f^\alpha |^2 \right)\\
	&=\tromx{\fomy} \guklb \guijb (\deli \delk f^\alpha ) \overline{(\delj \dell f^\beta)}.
\end{align*}
Here, we used the Cauchy-Schwarz inequalities.
\end{proof}

The next proposition is the so-called  ``Jeffres' trick''.

\begin{proposition}[{\citep[Section 4]{Jeffres00}}]\label{JeffresTrick}
Let $X$ be a compact \kahler \ manifold, $D$ be a smooth divisor, and $\beta $ be a real number satisfying $0<\beta <1$. 
Let $s \in H^0(X, \mathcal{O}_X(D))$ be a holomorphic section of the line bundle $\mathcal{O}_X(D)$ whose zero divisor is $D$, and $h$ is a smooth Hermitian metric. 
Then, for any function $u \in C^{,\alpha, \beta}$ and $\varepsilon >0$, every maximum point of the function
\begin{align*}
	u_\varepsilon \deq u + \varepsilon |s|_h ^{2\gamma }
\end{align*}
on $X$ belongs to $X \setminus D$ if $0<2\gamma <\alpha \beta$.
\end{proposition}

\begin{proof}
We assume that $u_\delta $ takes maximum at $x_0 \in D$. 
Let $(U, (z^1, \dots, z^n))$ be a holomorphic chart centered at $x_0$ satisfying $D\cap U =\{ z^1 =0 \}$. 
By the definition of  $x_0$, for any $x = (z, 0,\dots, 0) \in U$, we have
\begin{align*}
	\dfrac{|u(x) - u(x_0)|}{d_\beta (x,x_0)^\alpha }
	=\dfrac{|u(x) - u(x_0)|}{|z|^{\alpha \beta}}
	\ge \dfrac{\varepsilon |s|_h ^{2\gamma } (x) }{|z|^{\alpha \beta }}
	\ge \dfrac{\varepsilon }{C}\dfrac{|z|^{2\gamma } }{|z|^{\alpha \beta }}.
\end{align*}
Since $0<2\gamma <\alpha \beta$, the right hand side goes to $\infty $ as $z\rightarrow 0$.
This contradicts with the definition of $ C^{,\alpha, \beta}$.
\end{proof}

Theorem \refs{VolThm} and Theorem \refs{TraceThm} can be shown in a smilar manner. We only prove Theorem \refs{VolThm} here.

\begin{proof}[Proof of Theorem \refs{VolThm} (a)] 
Since $f$ can be  represented as $(w^1, \dots, w^n) = ((z^1)^k, f_2(z), \dots,$ $f_n(z))$ such that $D=\{z^1=0 \}$ and $E=\{w^1=0\}$, the direct computation gives that $f$ is locally \holder \ continuous with respect to $d_\alpha $ and $d_\beta $ if $\alpha \le k\beta$. 
Combining with (\refs{vollocal}) and the definition of the cone metrics, $v\deq \fomy^n/\omx^n$ is a $C^{, \sigma , \beta }$ function for some $0<\sigma <1$.  
By Proposition \refs{JeffresTrick}, all maximum points of $v_\delta \deq v + \varepsilon |s|_h^{2\gamma }$ belong to $X \setminus D$ where $\gamma$ is sufficiently small.
Since $v_\varepsilon$ is smooth on $X\setminus D$, we can apply the maximum principle argument to $v_\varepsilon$. 
The direct computation show that
\begin{align*}
	\ddb |s|_h^{2\gamma } 
		&= \ddb e^{\gamma \log |s|_h^2}
		   =  |s|_h^{2\gamma } 
			   (
				   \gamma \ddb \log |s|_h^2 
				   + \gamma^2 \ii \partial \log |s|_h^2 \wedge \overline{\partial} \log |s|_h^2 
			   ) \\
		&\ge - \gamma |s|_h^{2\gamma } \ii R_h.
\end{align*}
Therefore, there exists a constant $C >0 $ (which is independent of $\varepsilon$) satisfying
\begin{align*}
	\lapomx |s|_h^{2\gamma } \ge -  \, C.
\end{align*}

Let $x_0 \in X\setminus D$ be a maximum point of $v_\varepsilon$.  At this point, by Proposition \refs{LapProp} (a), we have
\begin{align*}
	0 \ge \lapomx v_\varepsilon = \lapomx v + \varepsilon \lapomx |s|_h^{2\gamma }  \ge  v ( nB  v^{1/n} - A  ) -\varepsilon C.
\end{align*}
Simple calculus show that the function $t \mapsto t^n(nBt-A) -\varepsilon C $ takes non-positive values exactly on some bounded interval $[0, T_\varepsilon]$ and $T_\varepsilon \rightarrow A/(nB)$ as $\varepsilon \rightarrow 0$. 
It follows that
\begin{align*}
	v_\varepsilon(x_0) = v(x_0) + \varepsilon |s|_h^{2\gamma }(x_0)  \le T_\varepsilon^{n} + \varepsilon \sup_X |s|_h^{2\gamma }.
\end{align*}
Since the right hand side does not depend on $x_0$ and $x_0$ is any maximum point of $v_\varepsilon$, this inequality holds on whole $X$. 
Therefore, we have the following inequality 
\begin{align*}
	v= v_\varepsilon - \varepsilon |s|_h^{2\gamma }  \le v_\varepsilon \le T_\varepsilon^{n} + \varepsilon \sup_X |s|_h^{2\gamma }
\end{align*}
on $X$. By taking $\varepsilon\rightarrow 0$, we obtain $v\le (A/(nB))^{n}$.
\end{proof}

\begin{proof}[Proof of Theorem \refs{VolThm} (b)]
 By definition of the cone metric, we can easily see that for any $\varepsilon>0$, 
\begin{align*}
	\ve 
		\deq |s|_h^{2(\ell+ \varepsilon)}  v = |s|_h^{2(\ell+ \varepsilon)}   \dfrac{\fomy^n}{\omx^n}
\end{align*}
tends to $0$ as $x$ approaches to $D$, where $\ell \deq \alpha - k \beta>0$. 
Then, combining the Laplacian estimate in Proposition \refs{LapProp} (a), we have 
\begin{align*}
\lapomx \log \ve 
	&= -  (\ell +\varepsilon ) \tromx{ \ii R_h } + \lapomx \log v  \\
	&\ge -(\ell +\varepsilon) C  - A  + nB  v^{1/n},  \\
\lapomx \ve 
	&\ge \ve ( -(\ell+\varepsilon) C  - A  + nB v^{1/n} ) .
\end{align*}
If $x_0 \in X$ is a maximum of $\ve $, we can assume that $x_0 \in X\setminus D$. At this point, by applying the maximum principle, we have 
\begin{align*}
	v (x_0)
		&\le \left( \dfrac{A+(\ell+\varepsilon)C }{nB} \right)^n.
\end{align*}
Therefore, we get
\begin{align*}
\ve (x_0)
		& \le  |s|_h^{\ell+ \varepsilon}(x_0) \left( \dfrac{A+(\ell+\varepsilon)C }{nB} \right)^n 
		\le \left( \dfrac{A+(l+\varepsilon)C }{nB} \right)^n.
\end{align*}
Since the right hand side does not depend on $x_0$, this inequality holds on $X$. Taking $\varepsilon\rightarrow 0$, we obtain
\begin{align*}
	|s|_h^{2\ell} \dfrac{\fomy^n}{\omx^n}\le \left( \dfrac{A+\ell C }{nB} \right)^n.
\end{align*}
\end{proof}

\nocite{*} 
\bibliographystyle{amsalphaurlmod}
{\footnotesize
\bibliography{reference}

\begin{thebibliography}{CDS15b}

\bibitem[Ahl38]{Ahlfors39Schwarz}
L.~V. Ahlfors, \emph{An extension of {S}chwarz's lemma}, Trans. Amer. Math.
  Soc. \textbf{43} (1938), no.~3, 359--364, \doilink{10.2307/1990065},
  \MRlink{1501949}.

\bibitem[CDS15a]{CDSI}
X.~X. Chen, S.~Donaldson, and S.~Sun, \emph{K\"ahler-{E}instein metrics on
  {F}ano manifolds. {I}: {A}pproximation of metrics with cone singularities},
  J. Amer. Math. Soc. \textbf{28} (2015), no.~1, 183--197,
  \doilink{10.1090/S0894-0347-2014-00799-2}, \MRlink{3264766}.

\bibitem[CDS15b]{CDSII}
X.~X. Chen, S.~Donaldson, and S.~Sun, \emph{K\"ahler-{E}instein metrics on
  {F}ano manifolds. {II}: {L}imits with cone angle less than {$2\pi$}}, J.
  Amer. Math. Soc. \textbf{28} (2015), no.~1, 199--234,
  \doilink{10.1090/S0894-0347-2014-00800-6}, \MRlink{3264767}.

\bibitem[CDS15c]{CDSIII}
X.~X. Chen, S.~Donaldson, and S.~Sun, \emph{K\"ahler-{E}instein metrics on
  {F}ano manifolds. {III}: {L}imits as cone angle approaches {$2\pi$} and
  completion of the main proof}, J. Amer. Math. Soc. \textbf{28} (2015), no.~1,
  235--278, \doilink{10.1090/S0894-0347-2014-00801-8}, \MRlink{3264768}.

\bibitem[Che68]{Chern68}
S.-S. Chern, \emph{On holomorphic mappings of hermitian manifolds of the same
  dimension.}, Entire {F}unctions and {R}elated {P}arts of {A}nalysis ({P}roc.
  {S}ympos. {P}ure {M}ath., {L}a {J}olla, {C}alif., 1966), Amer. Math. Soc.,
  Providence, R.I., 1968, pp.~157--170, \MRlink{0234397}.

\bibitem[Don12]{Donaldson12KametricConesing}
S.~Donaldson, \emph{K\"ahler metrics with cone singularities along a divisor},
  Essays in mathematics and its applications, Springer, Heidelberg, 2012,
  pp.~49--79, \doilink{10.1007/978-3-642-28821-0\_4}, \MRlink{2975584}.

\bibitem[Jef00a]{Jeffres00}
T.~D. Jeffres, \emph{Schwarz lemma for {K}\"ahler cone metrics}, Internat.
  Math. Res. Notices (2000), no.~7, 371--382,
  \doilink{10.1155/S1073792800000210}, \MRlink{1749739}.

\bibitem[Jef00b]{MR1800816}
T.~D. Jeffres, \emph{Uniqueness of {K}\"ahler-{E}instein cone metrics}, Publ.
  Mat. \textbf{44} (2000), no.~2, 437--448,
  \doilink{10.5565/PUBLMAT\_44200\_04}, \MRlink{1800816}.

\bibitem[Lu68]{Lu68}
Y.-C. Lu, \emph{Holomorphic mappings of complex manifolds}, J. Differential
  Geometry \textbf{2} (1968), 299--312, \MRlink{0250243}.

\bibitem[Tia15]{Tian}
G.~Tian, \emph{K\"ahler-{E}instein metrics on {F}ano manifolds}, Jpn. J. Math.
  \textbf{10} (2015), no.~1, 1--41, \doilink{10.1007/s11537-014-1387-3},
  \MRlink{3320994}.

\bibitem[Yau78]{Yau78Schwarz}
S.-T. Yau, \emph{A general {S}chwarz lemma for {K}\"ahler manifolds}, Amer. J.
  Math. \textbf{100} (1978), no.~1, 197--203, \MRlink{0486659}.

\end{thebibliography}
}
\end{document}